\documentclass{amsart}
\usepackage{amssymb}
\usepackage{amsmath}
\usepackage{thmtools, thm-restate}
\usepackage{hyperref}
\hypersetup{
     colorlinks=true,
     linkcolor=blue,
     citecolor = blue,      
     }
\usepackage{cleveref}
\crefname{lem}{le lemme}{les lemmes}
\newtheorem{thm}{Theorem}[section]
\newtheorem{prop}[thm]{Proposition}
\newtheorem{lem}[thm]{Lemma}
\newtheorem{cor}[thm]{Corollary}
 
\theoremstyle{definition}
\newtheorem{definition}[thm]{Definition}
\newtheorem{remark}[thm]{Remark}
\newtheorem{example}[thm]{Example}

\numberwithin{equation}{section}

\newcommand{\M}{{\rm{M}}}

\newcommand{\F}{{\rm{F}}}

\begin{document}

\title{Eckardt points on a cubic threefold}
\author{Gloire Grâce Bockondas}
\address{Gloire Grâce Bockondas, Département de Mathématiques, Université Marien Ngouabi, Brazzaville, Congo}
\email{gloire.bockondas@umng.cg}
\urladdr{https://sites.google.com/view/gloiregbockondas/} 
\author{Basile Guy Richard Bossoto}
\address{Basile Guy Richard Bossoto, Département de Mathématiques, Université Marien Ngouabi, Brazzaville, Congo}
\email{basile.bossoto@umng.cg}

\begin{abstract}
In this paper we survey Eckardt points on a smooth complex cubic threefold with an approach aimed at computing all Eckardt points of a cubic threefold. In addition, we construct cubic threefolds with no Eckardt points but containing triple lines.
\end{abstract}
\maketitle

\section{Introduction} 
Eckardt points originate from a paper of F.E. Eckardt \cite{eckardt1876ueber}. They have been thoroughly studied in the case of cubic surfaces in $\mathbb{P}^{3}$, defined as points corresponding to the intersection of three of the 27 lines \cite{segre1943non}. They have then been generalized to higher-dimensional and higher-degree hypersurfaces \cite{cools2010star}. They are also called star points or inflection points \cite{tjurin1971geometry}. On a smooth complex cubic threefold $X\subset\mathbb{P}^{4}$, an Eckardt point $p\in X$ is a point for which the intersection $X\cap T_{p}X$ of the projective tangent space of $X$ at $p$ with $X$ has multiplicity three at $p$. This is equivalent to saying that the intersection $X\cap T_{p}X \subset T_{p}X$ is a cone with vertex $p$ over an elliptic curve $E_{p}$ \cite{cools2010star}. Each Eckardt point $p\in X$ parametrizes thus an elliptic curve $E_{p}$ on the Fano surface of lines $\F(X)$ of $X$, which is the base of the cone $X\cap T_{p}X$ (see \cite{tjurin1971geometry}). A cubic threefold can contain at most finitely many Eckardt points, and in fact at most 30, which is achieved by the Fermat cubic whereas the general one has none \cite{clemens1972intermediate, canonero1997inflection}. There are then at most 30 elliptic curves on the Fano surface $\F(X)$ while for the general cubic threefold there are none \cite{roulleau2009elliptic}, the Fano surface of the Fermat cubic threefold being the only one that contains exactly 30 elliptic curves. This is the most common characterization of Eckardt points on a cubic threefold in the literature.

Furthermore, Eckardt points on a smooth cubic hypersurface $Y\subset\mathbb{P}^{n}$ can be studied through polar quadrics. They have been intensively studied in this way in \cite{canonero1997inflection} where the authors found the maximal number of Eckardt points of a cubic hypersurface in $\mathbb{P}^{n}$. In \cite{cools2010star}, a connection between Eckardt points of a hypersurface of degree $d$ in $\mathbb{P}^{n}$ and polar hypersurfaces is used to determine all Eckardt points on the Fermat hypersurface of degree $d$ in $\mathbb{P}^{n}$.

Nevertheless, both equivalent characterizations and a method for finding all Eckardt points, certainly well-known to the experts, are difficult to find in the literature. This paper aims to fill this gap by studying Eckardt points on a cubic threefold using these two characterizations with an approach focusing on finding all Eckardt points of a cubic threefold. Moreover, we construct cubic threefolds with no Eckardt points but containing triple lines, which is as far as we know new. We also study through many examples the configuration of elliptic curves, triple lines, and the residual component of the union of elliptic curves in the curve $\M(X)$ of lines of the second type of $X$ called the main component. These computations show how elliptic curves, triple lines and the main component can be related in a cubic threefold.

\textbf{Acknowledgements}.
We wish to thank warmly Samuel Boissière for many useful discussions. Besides, we would like to thank Søren Gammelgaard and Yilong Zhang for interesting discussions. The first author has been supported by the Program EMS SIMONS for Africa and the ``Laboratoire de Mathématiques et Applications de l'Université de Poitiers UMR CNRS 7348".

\section{Notations and Preliminaries}
For $X\subset\mathbb{P}^{4}$ a smooth complex cubic threefold, the Fano surface $\F(X)$ is a smooth general type surface that parametrizes the lines on $X$ (see \cite{clemens1972intermediate}). Lines on $X$ are either of the first type or of the second type \cite{clemens1972intermediate} depending on the decomposition of their normal bundles. A line $\ell\subset X$ is said to be of the second type if and only if there exists a unique 2-plane $P\supset\ell$ tangent to $X$ in all points of $\ell$. We write $P\cap X=2\ell\cup\ell^{'}$, where $\ell^{'}$ is the residual line of $\ell$. Otherwise we say that $\ell$ is a line of the first type. For $\ell\neq\ell^{'}$ the line $\ell$ is called a double line, and if $\ell=\ell^{'}$ we say that $\ell$ is a triple line. The locus $\M(X)$ of lines of the second type on $X$ is a curve whose the singularities are exactly the points corresponding to triple lines on $X$ (see  \cite{bockondas2023triple}). However, this curve is smooth for a generic cubic threefold $X\subset\mathbb{P}^{4}$  \cite{huybrechts2020geometry}. 

Denote by $p_{i,j}$, $0\leq i<j\leq 4$, the Plücker coordinates of the grassmannian of lines $\mathbb{G}(1,4)\subset\mathbb{P}^{9}$. On the affine chart $p_{0,1}=1$ of $\mathbb{G}(1,4)$ with local coordinates $(p_{0,2},p_{0,3},p_{0,4},p_{1,2},p_{1,3},p_{1,4})$ we have the decomposition $f(p)=\displaystyle\sum_{i+j = 3}t_{0}^{i}t_{1}^{j}\phi^{i,j}(\ell)$ for any point $p\in\ell\subset X$ with coordinates $t_{0}v_{0}+t_{1}v_{1}$, where $\phi^{i,j}(\ell)$ are functions of the local Plücker coordinates of $\ell$ and $f=0$ the equation of $X$. The Fano surface $\F(X)$ is then defined by the vanishing locus of the terms $\phi^{i,j}(\ell)$. On the other hand, any 2-plane $P$ that contains $\ell$ meets the plane $\pi=\lbrace x_{0}=0, x_{1}=0\rbrace$ at a unique point $v_{2}=(0:0:\alpha_{2}:\alpha_{3}:\alpha_{4})$ such that $\ell$ and $v_{2}$ span $\pi$. The plane cubic $P\cap X$ is then defined by $f(t_{0}v_{0} + t_{1}v_{1} + t_{2}v_{2})=0$ where $(v_{0}: v_{1}: v_{2})$ are the projective
coordinates of $P$. Expanding in $t_{2}$ we write:

\begin{equation*}
0=f(t_{0}v_{0}+t_{1}v_{1}) + t_{2} \sum_{i=2}^{4}\dfrac{\partial f}{\partial x_{i}}(t_{0}v_{0}+t_{1}v_{1})\alpha_{i}+\dfrac{1}{2}t_{2}^{2}\sum_{2\leq i,j\leq 4}\dfrac{\partial^{2} f}{\partial x_{j}\partial x_{i} }(t_{0}v_{0}+t_{1}v_{1})\alpha_{i}\alpha_{j}+ t_{2}^{3}f(v_{2}).
\end{equation*}
The line $\ell\subset P$ of equation $t_{2}=0$ is a line of the second type on $X$ if and only if $f(t_{0}v_{0}+t_{1}v_{1})=0$ and the plane cubic equation is a multiple of $t_{2}^{2}$. Furthermore, the second type line $\ell\subset X$ of equation $t_{2}=0$ is a triple line if and only if the plane cubic equation is a multiple of $t_{2}^{3}$ (see \cite{bockondas2023triple}).

\section{Characterizations of Eckardt points on a cubic threefold}

\subsection{Eckardt points and elliptic curves}
We recall the definition of an Eckardt point on a smooth complex cubic threefold $X\subset\mathbb{P}^{4}$ (see \cite[Definition 1.5]{laza2018moduli} and \cite[Proposition 6.3.5]{gammelgaard2018cubic}). Denote by $T_{p}X$ the projective tangent space of $X$ at $p\in X$.

\begin{definition}\label{ECKARDT}
A point $p\in X$ is an Eckardt point if it is a point of multiplicity three for the cubic $X\cap T_{p}X\subset T_{p}X$.
\end{definition}
Choose coordinates $(x_{0}:\ldots: x_{4})\in\mathbb{P}^{4}$ such that $p = (1 : 0 : 0 : 0 : 0)$ and $T_{p}X=\lbrace x_{1}=0\rbrace$. The equation of $X$ may be written 

\begin{equation*}\label{one}
f(x_{0},\ldots,x_{4}) = x_{0}^{2}x_{1} + x_{0}Q(x_{1},\ldots,x_{4}) + C(x_{1},\ldots,x_{4})
\end{equation*}
where $ Q(x_{1},\ldots,x_{4})$ and $C(x_{1},\ldots,x_{4})$ are homogeneous polynomials of degree two and three respectively. So if $p\in X$ is an Eckardt point then $ Q(x_{1},\ldots,x_{4})=0$ and the equation of $X$ may take the form

\begin{equation}\label{one}
f(x_{0},\ldots,x_{4}) = x_{0}^{2}x_{1} + C(x_{1},\ldots,x_{4}).
\end{equation}
Following \cite[p.169-170]{murre1972algebraic} (see also \cite[Proposition 6.3.5]{gammelgaard2018cubic}) we have the following proposition.
\begin{prop}
A point $p\in X$ is an Eckardt point if and only if it is contained in infinitely many lines on $X$.
\end{prop}
\begin{proof}
Consider a line $\ell$ going through $p$. It cuts out the hyperplane $x_{0}=0$ in a unique point $q\in\mathbb{P}^{4}$;
every point on $\ell$ has coordinates $\lambda p + \mu q$ with $(\lambda:\mu)\in\mathbb{P}^{1}$. The line $\ell$, defined by $x_{0}=\lambda, x_{i}=\mu q_{i}$ with $i=1,\ldots,4$, lies on $X$ if and only if $f(\lambda p + \mu q)=\lambda^{2}\mu q_{1}+\lambda\mu^{2}Q(q_{1},\ldots,q_{4})+\mu^{3}C(q_{1},\ldots,q_{4})=0$ for all $(\lambda:\mu)\in\mathbb{P}^{1}$, that is if and only if
$$q_{1}=0,\quad Q(q_{1},\ldots,q_{4})=0,\quad C(q_{1},\ldots,q_{4})=0.$$
The lines $\ell\subset X$ through $p$ correspond thus to the points $(x_{2}:x_{3}:x_{4})\in\mathbb{P}^{2}$  satisfying the equations $Q(0,x_{2},x_{3},x_{4})=0$ and $C(0,x_{2},x_{3},x_{4})=0$, that is the intersection points of a conic and a cubic in the plane of equation $\lbrace x_{0}=0,x_{1}=0\rbrace$. If $p$ is an Eckardt point then $Q(0,x_{2},x_{3},x_{4})=0$. The intersection $X\cap T_{p}X$ is a cone with vertex $p$ over the elliptic curve $E_{p}$ of equation $\left\lbrace x_{0}=0, C(0,x_{2},x_{3},x_{4})=0\right\rbrace$; the point $p$ is then contained in infinitely many lines on $X$. Conversely, if $Q(0,x_{2},x_{3},x_{4})$ and $C(0,x_{2},x_{3},x_{4})$ have a common factor there are infinitely many lines through $p$ contained in $X$, otherwise there are six lines in $X$ going through $p$. Moreover, if this common factor is linear then $X$ contains a plane and, if it is quadratic $X$ contains a quadratic cone, and hence a plane; this is impossible because of the smoothness of $X$. Therefore  $Q(0,x_{2},x_{3},x_{4})=0$ and $p$ is an Eckardt point.
\end{proof}
Every Eckardt point $p\in X$ parameterizes an elliptic curve $E_{p}\subset\F(X)$ of equation $\left\lbrace x_{0}=0, C(0,x_{2},x_{3},x_{4})=0\right\rbrace$, the base of the cone $X\cap T_{p}X$, and inversely every elliptic curve gives rise to an Eckardt point \cite{tjurin1971geometry, roulleau2009elliptic}. 
Moreover, there are at most finitely many Eckardt points on a smooth cubic threefold whereas a general one has no Eckardt points \cite[Lemma 2.7]{zhang2023extension}. The following result has been proven in \cite{gammelgaard2018cubic}.
\begin{prop}\cite[p.315]{clemens1972intermediate}\label{30}
A cubic threefold can contain at most 30 Eckardt points.
\end{prop}
\begin{proof}
We reproduce the proof and for completeness see \cite[Proposition 6.3.8]{gammelgaard2018cubic} and \cite[Lemma 4.8.4]{zhang2022topological}. Consider $P\subset\mathbb{P}^{4}$ a plane and $K_{\F}:=\lbrace [\ell]\in\F(X)\vert \ell\cap P\neq 0\rbrace$ the canonical divisor of $\F(X)$ (see \cite{clemens1972intermediate}). Let $C_{\ell}$ denote the curve of lines on $X$ incident to~$\ell$. We have $K_{\F}\cdot E_{p}=3$, then $C_{\ell}\cdot E_{p}=1$ since $K_{\F}=3C_{\ell}$ \cite[(10.9)]{clemens1972intermediate}. On the other hand, any component of $\M(X)$ intersects $K_{\F}$ non-negatively since $K_{\F}$ is effective. Moreover, all elliptic curves $E_{p_{i}}$ are parametrised by Eckardt points $p_{i}$ and are contained in $\M(X)$. There are thus at most $C_{\ell}\cdot \M(X)$ Eckardt points in $X$, with $C_{\ell}\cdot \M(X)=2C_{\ell}\cdot K_{\F}=6C_{\ell}^{2}=30$ since $\M(X)=
2K_{\F}$ and $C_{\ell}^{2}=5$
by \cite[Proposition 10.21, (10.8)]{clemens1972intermediate}.
\end{proof}
The Fano surface $\F(X)$  contains therefore at most 30 elliptic curves. Note that the curve $\M(X)$ of lines of the second type may contain components other than elliptic curves. However, if it contains exactly 30 elliptic curves then it has no components besides the elliptic components. Only one cubic hypersurface of $\mathbb{P}^{4}$ has 30 Eckardt points: the Fermat cubic $F_{4}$. Its Fano surface $\F(F_{4})$ is the only Fano surface that contains 30 elliptic curves \cite{ roulleau2009elliptic}.

\begin{definition}
The residual component of the union of elliptic curves in the curve of lines of the second type is called the main component.
\end{definition}
Apart from the Fermat cubic, for every smooth cubic threefold $X\subset\mathbb{P}^{4}$ containing Eckardt points the main component is not empty.

\begin{prop}\cite[Lemma 1.18]{murre1972algebraic}
If $\ell\subset X$ is a line of the first type and $p\in\ell$ then there are six lines on $X$ through $p$.
\end{prop}
Every point $p\in X$ not contained in a line of the second type is thus contained in six lines and we have the following proposition.

\begin{prop}\cite{murre1972algebraic}
A line going through an Eckardt point is of the second type.
\end{prop}
We give an elementary proof in coordinates of the following theorem.

\begin{thm}\cite{tjurin1971geometry}\label{Tjurin}
Let $p\in X$ be an Eckardt point.
The triple lines on $X$ correspond exactly to the inflection points of the elliptic curve $E_{p}$ which is the base of the cone $X\cap T_{p}X\subset T_{p}X$.
\end{thm}
\begin{proof}
Let $a\in E_{p}$ be a point and $\ell_{a}$ the tangent line of $E_{p}$ at $a$. Then $\ell_{a}$ is defined by $\displaystyle\sum_{i=2}^{4}
x_{i}\dfrac{\partial C}{\partial  x_{i}}(a)=0.$
The 2-plane $P_{1}$ spanned by $p$ and $\ell_{a}$ is tangent to $X$ along all of $\ell$. Let $b\in \ell_{a}$ be a point such that $b\neq a$ and $P_{2}$ the 2-plane in which lies $E_{p}$. Since $\ell_{a}\subset P_{2}$ then $b\in P_{2}$ and one can write 

\begin{equation}\label{2.10}
\displaystyle\sum_{i=2}^{4} b_{i}\dfrac{\partial C}{\partial x_{i}}(a) =0.
\end{equation}
We have thus $P_{1}={\rm{span}}(a,b,p)$ and the plane cubic $P_{1}\cap X$ is defined by $f(t_{0}p + t_{1}a + t_{2}b)=0$ with $(t_{0}:t_{1}:t_{2})$ the projective
coordinates of $P_{1}$. Expanding in $t_{2}$ and using Equations \eqref{one} and \eqref{2.10}, one can see that the line $\ell$ of equation $t_{2}=0$ is a double line on $X$. This second type line is a triple line if and only if 

\begin{equation}\label{2.50}
t_{1}\sum_{i=2}^{4}\dfrac{\partial^{2} C}{\partial x_{i}^{2}}\left(a\right) b_{i}^{2}+2t_{1}\displaystyle\sum_{2\leq i<j\leq 4}\dfrac{\partial^{2}C}{\partial x_{j}\partial x_{i}}\left(a\right) b_{i}b_{j}=0\quad\mbox{and}\quad  \dfrac{\partial^{3}C}{\partial t_{2}^{3}}(t_{1}a)\neq 0
\end{equation}
holds. Now we are going to study inflection points on the elliptic curve $E_{p}$. The point $a\in E_{p}$ is an inflection point if it is a point of multiplicity three for the intersection $E_{p}\cap \ell_{a}$ defined by $C(t_{1}a+t_{2}b)=0$. Since $C(t_{1}a)$ and $\dfrac{\partial C}{\partial t_{2}}(t_{1}a)$ vanish then $a\in E_{p}$ is an inflection point if and only if \eqref{2.50} holds, which are necessary and sufficient conditions for the line $\ell$ of equation $t_{2}=0$ to be a triple line on $X$.
\end{proof}
The planes $P_{1}$ and $P_{2}$ meet along $\ell_{a}$, the tangent line to $E_{p}$ at $a$. The point $a$ gives  thus rise to the line of the second type $\ell\subset X$, and conversely the line of the second type gives rise to the point $a$. When $a\in E_{p}$ is not an inflection point, the tangent line $\ell_{a}$ cuts out the elliptic curve $E_{p}$ in a third point $a^{'}\in E_{p}$ which gives rise to the residual line $\ell^{'}$ of the double line $\ell$.

\begin{cor}\label{converse}
If a smooth complex cubic threefold contains Eckardt points then it contains triple lines.
\end{cor}
\begin{example}\text{}
\begin{enumerate}
\item[1.] The Fermat cubic defined by $ x_{0}^{3} + x_{1}^{3} + x_{2}^{3} + x_{3}^{3} + x_{4}^{3}=0$ has 30 Eckardt points with coordinates $(0,\ldots,\underbrace{1}_{x_{i}},\ldots,\underbrace{\xi}_{x_{j}},\ldots,0)$ with $x_{k}=0$ for $k\neq i,j$ and $\xi\in\mathbb{C}$ such that $\xi^{3}=-1$, and  contains 135 triple lines.
\item[2.] The Klein cubic defined by $x_{0}^{2}x_{1} + x_{1}^{2}x_{2} + x_{2}^{2}x_{3} + x_{3}^{2}x_{4} + x_{4}^{2}x_{0}=0$ contains neither Eckardt points nor triple lines.
\item[3.] The cubic threefold defined by $ x_{0}^{2}x_{2} + x_{2}^{2}x_{4} + x_{1}^{2}x_{3} + x_{3}^{2}x_{0} + x_{4}^{3}=0$ has one Eckardt point with coordinates $(0:1:0:0:0)$ and contains 9 triple lines.
\end{enumerate}
\end{example}
\begin{remark}
Nevertheless, the converse of Corollary \ref{converse} is not true. There exist smooth complex cubic threefolds with no Eckardt points but containing triple lines (see Section \ref{confi}). 
\end{remark}
\begin{cor}
There are exactly nine triple lines going through an Eckardt point on a smooth cubic threefold. 
\end{cor}
Every smooth cubic threefold containing Eckardt points contains therefore at least nine triple lines.

\subsection{Eckardt points and polar quadrics}

Let $X\subset\mathbb{P}^{4}$ be a  smooth complex cubic threefold and $p\in X$ a point. We recall the following definition (see \cite[Definition 2.11]{cools2010star}). 
\begin{definition}
The polar quadric of a point $p=(p_{0}:\ldots:p_{4})\in\mathbb{P}^{4}$ with respect to $X$ is the hypersurface defined by $\displaystyle\sum_{i=0}^{4}p_{i}\dfrac{\partial f}{\partial x_{i}}=0$. 
\end{definition}
Denote by $\bigtriangleup_{p}(X)$ the polar quadric of $p$ with respect to $X$. From \cite[p.161-162]{canonero1997inflection} we have the following proposition.

\begin{prop}\label{Propo}
A point $p\in X$ is an Eckardt point if and only if the polar quadric $\bigtriangleup_{p}(X)$ splits up as the tangent space $T_{p} X$ and a hyperplane not passing through $p$.
\end{prop}
\begin{proof}
Let $p=(1:0:0:0:0)\in X$ be an Eckardt point and $T_{p}X=\left\lbrace x_{1}=0\right\rbrace$ be the projective tangent space of $X$ at $p$. Then the equation of $X$ can take the form 

\begin{equation*}
f(x_{0},\ldots,x_{4})=x_{0}^{2}x_{1} + C(x_{1},\ldots,x_{4})
\end{equation*}
where $C(x_{1},\ldots,x_{4})$ is a homogeneous polynomial of degree three. The polar quadric $\bigtriangleup_{p}(X)$ is defined by $x_{0}x_{1}=0$. It therefore contains the tangent space $T_{p}X$ and a hyperplane not passing through $p$. Conversely, choose coordinates on $\mathbb{P}^{4}$ such that $p=(1:0:0:0:0)$ is a point of $X$, $T_{p}X=\left\lbrace x_{1}=0\right\rbrace$ and $H_{p}X=\left\lbrace x_{0}=0\right\rbrace$ is a hyperplane not passing through $p$. The equation of $X$ can be written:

\begin{equation*}
f(x_{0},\ldots,x_{4})=x_{0}^{2}x_{1}+x_{0}Q(x_{1},\ldots,x_{4})+C(x_{1},\ldots,x_{4})
\end{equation*}
where $Q(x_{1},\ldots,x_{4})$ and $C(x_{1},\ldots,x_{4})$ are homogeneous polynomials of degree two and three respectively. The polar quadric $\bigtriangleup_{p}(X)$ is given by the equation:

\begin{equation*}\label{0307}
2x_{0}x_{1}+Q(x_{1},\ldots,x_{4})=0.
\end{equation*}
Since it splits up as $T_{p}X=\left\lbrace x_{1}=0\right\rbrace$ and $H_{p}X=\left\lbrace x_{0}=0\right\rbrace$ then $Q(x_{1},\ldots,x_{4})=0$ and $p$ is an Eckardt point.
\end{proof}
The two characterizations of Eckardt points studied in this paper are thus equivalent.
The following lemma shows how polar quadrics can be used to find all Eckardt points on a cubic threefold.

\begin{lem}\label{LELEMME}
A point $p\in X$ is an Eckardt point if and only if the polar quadric $\bigtriangleup_{p}(X)$ is of rank at most two.
\end{lem}
\begin{proof}
Let $p\in X$ be an Eckardt point, $T_{p}X=\left\lbrace l_{1}(x_{0},\ldots,x_{4})=0\right\rbrace$ the projective tangent space of $X$ at $p$ and $H_{p}X=\left\lbrace l_{2}(x_{0}, \ldots,x_{4})=0\right\rbrace$ a hyperplane not passing through $p$, with $l_{1}(x_{0},\ldots,x_{4})$ and $l_{2}(x_{0},\ldots,x_{4})$ two linear forms. Assume the polar quadric $\bigtriangleup_{p}(X)$ is defined by the equation $q(x_{0},\ldots,x_{4})=0$, where $q(x_{0},\ldots,x_{4})$ is a homogeneous polynomial of degree two. Using Proposition \ref{Propo} we write
$$q(x_{0},\ldots,x_{4})=l_{1}(x_{0},\ldots,x_{4})l_{2}(x_{0},\ldots,x_{4})$$
and the quadratic form is of rank at most two. Conversely, suppose $p\in X$ is not an Eckardt point. Then the quadratic form $q(x_{0},\ldots,x_{4})$ is not the product of two linear forms (see Proposition \ref{Propo}) and $\bigtriangleup_{p}(X)$ is therefore of rank at least three.
\end{proof}

\section{Computing Eckardt points on a cubic threefold}
We have studied Eckardt points on a cubic threefold through two different approaches: the first one involving elliptic curves and the second one involving polar quadrics. Both approaches can be used to compute Eckardt points on a cubic threefold. However, it is generally challenging to compute Eckardt points on a cubic threefold using the first approach because the expression of the tangent space $T_{p}X$ can make the computation of points of multiplicity three of $X\cap T_{p}X\subset T_{p}X$ difficult. It is therefore easier to check whether a point of the cubic is an Eckardt point than to find all Eckardt points using the first approach. Nevertheless, this approach has the benefit of revealing the equations of the elliptic curves of the Fano surface.

Unlike the first approach, the second one can be used to compute all Eckardt points of a cubic threefold and check whether a point on the cubic is an Eckardt point. It therefore has the advantage of revealing the number of Eckardt points of a cubic threefold.

Using Lemma \ref{LELEMME}, we propose the following method for computing all Eckardt points on a cubic threefold.

\subsection{Method for computing all Eckardt points on a cubic threefold.}
Let $X\subset\mathbb{P}^{4}$ be a smooth complex cubic threefold, $p=(p_{0}:\ldots:p_{4})\in X$ a point and $\mathcal{B}$ the matrix associated with the polar quadric $\bigtriangleup_{p}(X)$. Eckardt points on $X$ are given by the vanishing locus of all $3\times3$ minors of $\mathcal{B}$. In order to count each point only once, we use a stratification of $\mathbb{P}^{4}$ described as follows: the first stratum is the affine chart $p_{0}=1$ and the $i$-th one is defined by $p_{0}=0,\ldots,p_{i-2}=0, p_{i-1}=1$ for $i=1,\ldots,4$.

\begin{example}\text{}
\item[1.] Consider the Fermat cubic $F_{4}=\lbrace x_{0}^{3}+x_{1}^{3}+x_{2}^{3}+x_{3}^{3}+x_{4}^{3}=0\rbrace\subset\mathbb{P}^{4}$. Denote by $\mathcal{B}_{1}$ the matrix associated with the polar quadric $\bigtriangleup_{p}F_{4}$.

\begin{enumerate}

\item On the affine chart $p_{0}=1$ the vanishing locus of all $3\times 3$ minors of $\mathcal{B}_{1}$ is defined by the following equations:

\begin{center}
$p_2^4 + p_2=0,~p_3^4 + p_3=0,~p_4^4 + p_4=0,~p_1^3 + p_2^3 + p_3^3 + p_4^3 + 1=0,~p_1p_2=0,\newline p_1p_3=0,~p_2p_3=0,~ p_1p_4=0,~p_2p_4=0,~p_3p_4=0$.
\end{center}
If $p_{i}\neq 0$ then $p_{j}=0$ for $i\neq j$ and $p_{i}^{3}=-1$. We get 12 Eckardt points with coordinates $(1:0:\ldots:p_{i}:\ldots:0)$ with $p_{i}^{3}=-1$.
\item In the stratum $p_{0}=0, p_{1}=1$ the vanishing locus of all $3\times 3$ minors of $\mathcal{B}_{1}$ is defined by the following equations:

\begin{center}
$p_3^4 + p_3=0,~p_{4}^4 + p_4=0,~p_2^3 + p_3^3 + p_4^3 + 1=0,~p_2p_3=0,\newline
p_2p_4=0,~p_3p_4=0$.
\end{center}
If $p_{i}\neq 0$ then $p_{j}=0$ for $i\neq j$ and $p_{i}^{3}=-1$. We get 9 Eckardt points with coordinates $(0:1:\ldots:p_{i}:\ldots:0)$ with $p_{i}^{3}=-1$.
\item In the stratum $p_{0}=0, p_{1}=0, p_{2}=1$ the vanishing locus of all $3\times 3$ minors of $\mathcal{B}_{1}$ is defined by the following equations:

\begin{center}
$p_4^4 + p_4=0,~p_3^3 + p_4^3 + 1=0,~p_3p_4=0$.
\end{center}
We get three Eckardt points with coordinates $(0:0:1:\xi:0)$ and three others  with coordinates $(0:0:1:0:\xi)$, with $\xi^{3}=-1$.
\item In the last stratum we get $3$ Eckardt points with coordinates $(0:0:0:1:\xi)$ with $\xi^{3}=-1$.
\end{enumerate}
\item[2.] The Klein cubic contains no Eckardt point because the vanishing locus of all $3\times 3$ minors of the matrix associated with its polar quadric is empty in all strata.
\end{example}
Without giving much details, the authors in \cite[Example~ 4.4]{canonero1997inflection} state  that the cubic threefold $	X_{2}\subset\mathbb{P}^{4}$ defined by 
$$x_{0}^{2}x_{4} + x_{1}^{2}x_{3} + x_{3}^{3} + x_{3}^{2}x_{4} + x_{3}x_{4}^{2} - x_{4}^{3} + x_{2}^{3} = 0$$
has exactly two Eckardt points with coordinates $(1:0:0:0:0)$ and $(0:1:0:0:0)$. It is easy to check that these points are Eckardt points of $X_{2}$ using both approaches. However, showing that $X_{2}$ has no Eckardt points besides $(1:0:0:0:0)$ and $(0:1:0:0:0)$ is quite challenging. Therefore the importance of the method we proposed in this paper for computing all Eckardt points of a cubic threefold. Let $p=(p_{0}:p_{1}:p_{2}:p_{3}:p_{4})\in X_{2}$ be a point and denote by $\mathcal{B}_{2}$ the matrix associated with the polar quadric $\bigtriangleup_{p}X_{2}$. On the affine chart $p_{0}=1$ the vanishing locus of all $3\times 3$ minors of $\mathcal{B}_{2}$ is defined by $p_{1}=0,  p_{2}=0, p_{3}=0, p_{4}=0$, and in the stratum $p_{0}=0, p_{1}=1$ it is defined by $p_{2}=0, p_{3}=0, p_{4}=0$ while it is empty in the other strata. This proves that $X_{2}$ has no Eckardt points besides $(1:0:0:0:0)$ and $(0:1:0:0:0)$.

\section{Main component, elliptic curves and triple lines configuration of some cubic threefolds}\label{confi}

\subsection{Strategy for finding some cubic threefolds with no Eckardt points but containing triple lines}\label{subsection}
Let $\ell$ be a line of the second type on $X$ given by $$x_{2}=0, x_{3}=0, x_{4}=0.$$ Following \cite[(6.10)]{clemens1972intermediate} the equation of $X$ may take the form:

\begin{equation*}
f(x_{0},\ldots, x_{4}) = x_{0}^{2}x_{2} + x_{1}^{2}x_{3} + x_{0}q_{0}(x_{2},x_{3}, x_{4}) + x_{1}q_{1}(x_{2},x_{3}, x_{4}) + P(x_{2},x_{3}, x_{4})=0
\end{equation*}
where $q_{0}(x_{2},x_{3}, x_{4})=\sum_{2\leq j\leq k\leq 4}b_{0jk}$ and $q_{1}(x_{2},x_{3}, x_{4})=\sum_{2\leq j\leq k\leq 4}b_{1jk}$ are homogeneous polynomials of degree two and $P(x_{2},x_{3}, x_{4})$ is a homogeneous polynomial of degree three. Assume $\ell$ is a triple line so that the plane given by $x_{2}=0, x_{3}=0$ is the plane tangent to $X$ in all points of $\ell$. Then the equation of $X$ may be written

\begin{equation}\label{(5.1)}
f(x_{0},\ldots, x_{4}) = x_{0}^{2}x_{2} + x_{1}^{2}x_{3} + x_{0}q_{0}(x_{2},x_{3}, x_{4}) + x_{1}q_{1}(x_{2},x_{3}, x_{4}) + kx_{4}^{3}=0
\end{equation}
with $k\neq 0$ and $b_{044}=0, b_{144}=0$. Using Equation $\eqref{(5.1)}$ we obtain many examples of smooth cubic threefolds with no Eckardt points but containing triple lines. 

\subsection{Main component, elliptic curves and triple lines configuration}

The following table gives the list of smooth complex cubic threefolds $X_{i}\subset\mathbb{P}^{4}$ we will work with in this section. Cubics $X_{5}, X_{6}, X_{7}$ and $X_{8}$ are obtained through the method detailed in Section \ref{subsection}.

\begin{table}[h!]
\begin{tabular}{|l|l|}
\hline 
$X_{1}$ & $x_{0}^{2}x_{2} + x_{2}^{2}x_{4} + x_{1}^{2}x_{3} + x_{3}^{2}x_{0} + x_{4}^{3}=0$\\ 
\hline 
$X_{2}$ &  $2x_{0}x_{2}^{2} + 2x_{2}x_{1}^{2} + x_{1}^{2}x_{3} + x_{3}x_{0}^{2} + 3x_{3}^{3} + x_{4}^{3}=0$\\
\hline
$X_{3}$ & $x_{0}^{2}x_{4} + x_{1}^{2}x_{3} + x_{3}^{3} + x_{3}^{2}x_{4} + x_{3}x_{4}^{2} - x_{4}^{3} + x_{2}^{3} = 0$ \\ 
\hline
$ X_{4}$ & $x_{0}^{3} + x_{1}^{3} + x_{2}^{3} + x_{3}^{3} + x_{4}^{3} + 3x_{0}x_{1}x_{2}=0$\\
\hline 
$X_{5}$ & $x_{0}^{2}x_{2} + x_{1}^{2}x_{3} + x_{1}x_{2}^{2} + x_{0}x_{3}^{2} + x_{1}x_{3}^{2} + x_{4}^{3}=0$ \\  
\hline 
$X_{6}$ & $x_{0}^{2}x_{2} + x_{1}^{2}x_{3} + x_{0}x_{2}^{2} + x_{1}x_{2}^{2} + x_{0}x_{3}^{2} + 2x_{1}x_{2}x_{4} + 2x_{0}x_{3}x_{4} + x_{4}^{3}=0$ \\ 
\hline 
$X_{7}$ & $x_{0}^{2}x_{2} + x_{1}^{2}x_{3} + x_{1}x_{2}^{2} + x_{0}x_{3}^{2} + 2x_{0}x_{3}x_{4} + x_{4}^{3}=0$ \\ 
\hline 
$X_{8}$ & $x_{0}^{2}x_{2} + x_{1}^{2}x_{3} + x_{1}x_{2}^{2} +  x_{1}x_{3}^{2} + x_{0}x_{3}^{2} + x_{0}x_{4}^{2} + x_{1}x_{4}^{2} + x_{4}^{3}=0$\\
\hline 
\end{tabular} 
\caption{Some smooth cubic threefolds}
\end{table}
Table \ref{table2} gives the information about the number $n_{E}$ of Eckardt points and the number $n_{T}$ of triple lines of these cubics. 

\begin{table}[h!]
\begin{tabular}{|c|c|c|c|c|c|c|c|c|}
\hline 
cubic threefold & $X_{1}$ & $X_{2}$ & $X_{3}$ &$X_{4}$ &  $X_{5}$ & $X_{6}$ & $X_{7}$ & $X_{8}$ \\ 
\hline 
$n_{E}$ & 1 & 1 & 2 & 12 & 0 & 0 & 0 & 0\\ 
\hline 
$n_{T}$ & 9 & 33 & 39 & 81 & 27 & 9 & 2 & 1\\ 
\hline 
\end{tabular} 
\caption{Eckardt points and triple lines numbers}
\label{table2}
\end{table}
The following table gives the configuration of the main component, elliptic curves and triple lines for the above cubics. This table contains the number of triple lines, the number $n_{E_{p}}$ of elliptic curves, and the intersection number of the main component $P$ and elliptic curves in the affine chart $p_{0, 1}=1$. 

\begin{table}[h!]
\begin{tabular}{|c|c|c|c|c|c|}
\hline 
cubic threefold & $X_{1}$ & $X_{2}$ & $X_{3}$ &$X_{4}$ \\ 
\hline 
$n_{T}$ & 9 & 33  & 33 & 54  \\ 
\hline 
$n_{E_{p}}$ & 1 & 1 & 2 & 6  \\ 
\hline 
Intersection points & $Ep\cdot P =9$ & $Ep\cdot P =9$ & $Ep_{i}\cdot P =8$ & $Ep_{i}\cdot P =12$ (for  \\
number & & & $E_{p_{1}}\cdot E_{p_{2}} = 1$ &  3 elliptic curves) \\
& & & & $Ep_{j}\cdot P =6$ (for\\
& & & & the other elliptic curves)\\
& & & & $E_{p_{i}}\cdot E_{p_{j}} = 0$\\
\hline
\end{tabular} 
\caption{Main component, elliptic curves and triple lines configuration}
\label{table1}
\end{table}
The intersection points are computed over the rational field $\mathbb{Q}$. All the computations have been done using the software MAGMA \cite{MR1484478}, except the number of triple lines computed with SAGEMATHS \cite{sagemath}. The main component of $X_{1}, X_{2}, X_{3}$ and $X_{4}$ is irreducible over $\mathbb{Q}$. However, whether it is smooth is still an open question. Inspection of Table \ref{table1} reveals that the 9 intersection points of the elliptic curve and the main component of $X_{1}$ are exactly the triple lines of $X_{1}$.\\ 

\bibliographystyle{amsalpha}
\bibliography{biblio_Eckardt}

\end{document}